\documentclass[10pt]{amsart}
\usepackage{amssymb}
\usepackage[varbb, varg, smallerops]{newtxmath}
\usepackage[scaled=0.87]{inconsolata} 
\usepackage[T1]{fontenc}
\usepackage{amsfonts}
\usepackage{comment}
\usepackage[a4paper,margin=1.1in]{geometry} %MARGINI MIGLIORI
\usepackage{tikz}
\usepackage{tikz-cd}
\usetikzlibrary{arrows}
\tikzcdset{arrow style=tikz,
diagrams={>={Stealth[round,length=4pt,width=6pt,inset=3pt]}}}
\usepackage{graphicx}
\usepackage[all]{xy}
\makeatletter
\newcommand*{\tightdisplaymath}{\abovedisplayskip\z@\belowdisplayskip\z@}

\usepackage{calligra}

\makeatletter
\newcommand\RSloop{\@ifnextchar\bgroup\RSloopa\RSloopb}
\makeatother
\newcommand\RSloopa[1]{\bgroup\RSloop#1\relax\egroup\RSloop}
\newcommand\RSloopb[1]
  {\ifx\relax#1
   \else
     \ifcsname RS:#1\endcsname
       \csname RS:#1\endcsname
     \else
       \GenericError{(RS)}{RS Error: operator #1 undefined}{}{}
     \fi
   \expandafter\RSloop
   \fi
  }
\newcommand\X{0}
\newcommand\RS[1]
  {\begin{tikzpicture}
     [every node/.style=
       {circle,draw,fill,minimum size=1.5pt,inner sep=0pt,outer sep=0pt},
      line cap=round
     ]
   \coordinate(\X) at (0,0);
   \RSloop{#1}\relax
   \end{tikzpicture}
  }

\newcommand\RSdef[1]{\expandafter\def\csname RS:#1\endcsname}
\newlength\RSu
\RSdef{i}{\draw (\X) -- +(90:\RSu) node{};}
\RSdef{g}{\draw (\X) -- +(30:\RSu) node{};} %NUOVO MIO
\RSdef{l}{\draw (\X) -- +(135:\RSu) node{};}
\RSdef{r}{\draw (\X) -- +(45:\RSu) node{};}
\RSdef{I}{\draw (\X) -- +(90:\RSu) coordinate(\X I);\edef\X{\X I}}
\RSdef{G}{\draw (\X) -- +(25:\RSu) coordinate(\X G);\edef\X{\X G}} %NUOVO MIO
\RSdef{H}{\draw (\X) -- +(115:\RSu) coordinate(\X H);\edef\X{\X H}} %NUOVO MIO
\RSdef{L}{\draw (\X) -- +(135:\RSu) coordinate(\X L);\edef\X{\X L}}
\RSdef{R}{\draw (\X) -- +(45:\RSu) coordinate(\X R);\edef\X{\X R}}
\RSdef{D}{\draw (\X) -- +(35:\RSu) coordinate(\X D);\edef\X{\X D}} %NUOVO MIO
\RSdef{A}{\draw (\X) -- +(65:\RSu) coordinate(\X A);\edef\X{\X A}} %NUOVO MIO
\RSdef{B}{\draw (\X) -- +(115:\RSu) coordinate(\X B);\edef\X{\X B}} %NUOVO MIO
\RSdef{C}{\draw (\X) -- +(145:\RSu) coordinate(\X C);\edef\X{\X C}} %NUOVO MIO
\RSdef{d}{\draw (\X) -- +(35:\RSu) node{};} %NUOVO MIO
\RSdef{a}{\draw (\X) -- +(65:\RSu) node{};} %NUOVO MIO
\RSdef{b}{\draw (\X) -- +(115:\RSu) node{};} %NUOVO MIO
\RSdef{c}{\draw (\X) -- +(145:\RSu) node{};} %NUOVO MIO

\newcommand{\MDu}{\RS{{R{r}}{L{l}}}}

\newcommand{\MCS}{\RS{{R{r}}{L{l}r}}} %COMPOSIZIONE SINISTRA
\newcommand{\MCD}{\RS{{R{l}{r}}{L{l}}}} %COMPOSIZIONE DESTRA
\newcommand\restr[2]{{% we make the whole thing an ordinary symbol
  \left.\kern-\nulldelimiterspace % automatically resize the bar with \right
  #1 % the function
  \littletaller % pretend it's a little taller at normal size
  \right|_{#2} % this is the delimiter
  }}

\usepackage{amsmath}

\newcommand{\vers}{A model structure on $\aCatSt$}% $\aCat$}
\title[\vers]{A model structure on the category of $\Ain$-categories with strict morphisms}% $\aCat$}

\usepackage{hyperref}
\hypersetup{colorlinks=false}

\author{Mattia Ornaghi}
\address{\parbox{0.9\textwidth}{Universit\`a degli Studi di Milano\\
Dipartimento di Matematica\\
Via Cesare Saldini 50, 20133 Milano, Italy}}
\email{mattia12.ornaghi@gmail.com}

\theoremstyle{definition}
\newtheorem{defn}{Definition}[section]

\newtheorem{thm}{Theorem}[section]

\newtheorem{lem}[thm]{Lemma}

\newtheorem{cor}[thm]{Corollary}

\newtheorem*{namedthmA}{Theorem A}

\theoremstyle{remark}
\newtheorem{rem}{Remark}[section]
\newtheorem{exmp}{Example}[section]

\newcommand{\Ain}{\mbox{A$_{\infty}$}}

\newcommand{\aCat}{\mbox{A$_{\infty}$Cat}}
\newcommand{\aCatSt}{\aCat_{\tiny\mbox{strict}}}

\newcommand{\Id}{\mbox{Id}}

\newcommand{\A}{\mathscr{A}}

\newcommand{\B}{\mathscr{B}}

\newcommand{\Ho}{\mbox{Ho}}

\newcommand{\F}{\mathscr{F}}

\newcommand{\C}{\mathscr{C}}

\subjclass[2020]{14F08, 18G70, 18N40}

\thanks{The author was supported by the research project FARE 2018 HighCaSt (grant number R18YA3ESPJ) and by ERC Advanced
Grant-101095900-TriCatApp}

\begin{document}

\date{\today}

\maketitle

\begin{abstract}
We prove that the category of (strictly unital) $\Ain$-categories, linear over a commutative ring $R$, with strict $\Ain$-morphisms has a cofibrantly generated model structure.\ 
In this model structure every object is fibrant and the cofibrant objects have cofibrant morphisms.\ 
As a consequence we prove that the semi-free $\Ain$-categories (resp. resolutions) %defined in \cite{Orn2} 
are cofibrant objects (resp. resolution) in this model structure.
\end{abstract}

\section{Introduction-Statement of Results}
We fix a commutative ring $R$, an $\Ain$-category is a $R$-linear DG-category associative up to homotopy.\ 
In a few words, an $\Ain$-category $\A$ is a graded category equipped with multilinear maps 
\begin{align*}
m^n_{\A}:\A(x_{n-1},x_n)\otimes...\otimes\A(x_0,x_1)\to \A(x_0,x_n)[2-n],
\end{align*}
for every integer $n\ge 1$ and sequence of objects $x_0,,...,x_{n}\in\A$, satisfying axioms (see \cite[Definition 1.1.1 (1.1)]{Orn1}).\
Considering $m^1_{\A}$ as the differential, we can associate to $\A$ the (graded) category $H(\A)$ whose hom spaces are given by:
\begin{align}
H(\A)(x,y):=\bigoplus_{n\in\mathbb{Z}}H^n(\A (x,y)).
\end{align}

Taking two $\Ain$-categories $\A$ and $\B$, we call \emph{$\Ain$-functor} a family of multilinear maps $\mathcal{f}\F^n \mathcal{g}_{n\ge 0}$ of the form 
\begin{align}
\F^n:\A(x_{n-1},x_n)\otimes...\otimes\A(x_0,x_n)\to\B(\F^0(x_{n}),\F^0(x_0))[1-n]
\end{align}
for every integer $n\ge 1$ and sequence of objects $x_0,...,x_{n}\in\A$, satisfying axioms (see \cite[Definition 1.2.1 (1.2)]{Orn1}).\
There are several notions of unit in the framework of $\Ain$-categories, in this paper we consider the strictly unital $\Ain$-categories with strictly unital $\Ain$-functors (see \cite[Definitions 1.1.4]{Orn1}).\
From now on, we denote by $\aCat$ the category of $\Ain$-categories with the $\Ain$-functors.\
Note that $\aCat$ is not complete (since it does not admit equalizers \cite[Lemma 1.28]{COS1}).\\
We say that an $\Ain$-functor $\F$ is a \emph{quasi-equivalence} if it induces an equivalence of (graded) categories:
\begin{align}
[\F]:H(\A)\to H(\B)
\end{align}
and an equivalence $H^0(\F):H^0(\A)\to H^0(\B)$.\
We denote by $\Ho(\aCat)$ the homotopy category, i.e. the (Gabriel-Zisman) localization of $\aCat$ with respect to the class of quasi-equivalences.\\

$\Ain$-categories were introduced in the early 1990s in the context of Homological Mirror Symmetry since the Fukaya category of a symplectic manifold, comes naturally equipped with a structure of this kind.\
Note that the pioneers of this field, such as Kontsevich, Fukaya, Oh, Ono, Ohta, Seidel, Soibelman, etc, assumed by definition a \emph{flatness} hypothesis on the hom-spaces of an $\Ain$-category.\ 
For example, in \cite[Definition 1.1]{Fuk} and \cite[\S3.2.1]{FOOOH1} an $\Ain$-category $\A$ is such that the hom-spaces $\A(x,y)$ are graded free $R$-modules, or the base commutative ring $R$ is assumed to be a field, see \cite{KS} or \cite{Sei}.\

On the other hand, the definition of $\Ain$-category makes sense without any restriction on the hom-spaces.\ For this reason, more recently, many people such as Ganatra, Pardon, Shende, Oh, Tanaka started to used the term \emph{cofibrant $\Ain$-category} to indicate an $\Ain$-category whose hom-spaces are h-projective DG-modules (see \cite[Definition 2.6]{GPS}, \cite{Tan}, \cite[Definition 1.2]{OT}).\ Note that, despite the name, this has nothing to do with a model structure on $\aCat$, it is well known that $\aCat$ has no model structure \cite[\S1.5]{COS1}.\ 
The term \emph{cofibrant} is inherited by the DG-categories.\ Indeed a cofibrant DG-category (in the model structure of Example \ref{Tabboz}), has cofibrant hom-spaces.\ In particular it is a h-projective DG-category.\ 

The first goal of this paper is to give a precise definition of \emph{cofibrant $\Ain$-category}, namely we provided a model structure on the category of $\Ain$-categories (taking a subset of $\Ain$-functors).\ In this model structure, if an $\Ain$-category is cofibrant then it has cofibrant hom-spaces, in particular it is h-projective (see Theorem A).
\\

Despite the lack of a model structure, we can describe the hom-spaces of the homotopy category of $\aCat$ as follows: 
\begin{align}
\Ho(\aCat)(\A,\B)\simeq\Ho(\aCat)(\A^{\tiny\mbox{hps}},\B)/\approx. 
\end{align}
Here $\A^{\tiny\mbox{hps}}$ denotes a h-projective with splits unit $\Ain$-category which is quasi-equivalent to $\A$ and $\approx$ denotes the weakly equivalence relation \cite[Theorem A]{Orn2}.\\

In order to prove that every $\Ain$-category $\A$ has a resolution of the form $\A^{\tiny\mbox{hps}}$ was given the notion of semi-free $\Ain$-category, and it was proven that the semi-free resolution of $\A$ is a h-projective with splits unit $\Ain$-category (see \cite[Theorem 5.1]{Orn2}).\
Note that, in the case of DG-categories, the semi-free resolutions correspond to the cofibrant resolutions in Tabuada model structure (see Example \ref{Tabboz}).\ For this reason in \emph{loc.}\ \emph{cit.} it was left as an open question if the semi-free $\Ain$-categories are a kind of \emph{cofibrant resolutions} in an appropriate category.\ In this note we prove the following:

\begin{namedthmA}\label{lanzo}
There is a cofibrantly generated model structure on $\aCat_{\tiny\mbox{strict}}$ whose weak equivalences are the quasi-equivalences.\ In such a model structure the fibrations are the isofibrations (see Definition \ref{isofib}) strict $\Ain$-functors which are surjective on the morphisms.\ Every category is a fibrant object and every cofibrant object $\A$ is such that $\A(a_1,a_2)$ is a cofibrant object in $\mbox{Ch}(R)$.
\end{namedthmA}

Here $\aCatSt$ denotes the category of $\Ain$-categories with the strict $\Ain$-functors.\ Despite $\aCat$ is not complete, $\aCatSt$ is complete and cocomplete (see \cite[Theorem 4.5]{Orn2}).\ An $\Ain$-functor $\mathsf{F}=\mathcal{f}\mathsf{F}^n \mathcal{g}_{n=0,1}$ is called \emph{strict} $\Ain$-functor if $\mathsf{F}^0:\mbox{Ob}(\A)\to\mbox{Ob}(\B)$ and $\mathsf{F}^1$ preserves the underlying graded quivers, \emph{id est}:
\begin{align*}
\mathsf{F}^1(m^n_{\A}(a_1,...,a_n))=m^n_{\B}(\mathsf{F}^1(a_1),... ,\mathsf{F}^1(a_n))
\end{align*} 
for every $n\ge 1$ and $a_1,... ,a_n\in\A$.\ As a consequence of Theorem A, we have that the semi-free $\Ain$-categories are cofibrant (cf. \cite[Lemma D]{Orn2}).\\

We conclude by saying that, given an $\Ain$-category $\A$, we can take the cofibrant resolution $\A^{\tiny\mbox{cof}}$ (with respect to the model structure of Theorem A) which is cofibrant in the sense of Ganatra, Pardon, Shende \emph{et al}.

\subsection{State of art}

By the work of Lefevré-Hasegawa \cite{LH} the category $\mbox{Alg}_{\infty}$ of (non unital) $\Ain$-algebras linear over a field, has a model structure whose weak equivalences are the quasi-isomorphisms.\ In this model structure a morphism $f$ is a fibration (resp. cofibration) if $f^1$ is surjective (resp. injective).\

Note that $\mbox{Alg}_{\infty}$ has no equalizers (see \cite{COS1}) and coproduct.\ 
To see that $\mbox{Alg}_{\infty}$ has no coproduct we use the fact that (the bar-cobar functor) $\mbox{U}:\mbox{Alg}_{\infty}\to\mbox{DG-Alg}$ is right adjoint to the inclusion (see \cite{COS2}).\ Since the tensor product is the coproduct of two DG-algebras, we must have $\mbox{U}(A)\otimes \mbox{U}(B)\simeq\mbox{U}(A\otimes B)$, for any $\Ain$-algebras.\ Using this fact, it is easy to see that $\mbox{Alg}_{\infty}$ has no coproduct (it has to do with the fact that there is no a good notion of tensor product of $\Ain$-algebras \cite{Orn2}).\ This is no longer true if we consider the category of $\Ain$-categories, namely, it has coproducts (which is the disjoint union) but it does not have equalizers \cite[Lemma 1.28]{COS1}.\

On the other hand, if we consider the category of $\Ain$-algebras with strict $\Ain$-morphisms we have a model structure by \cite[2.2.1. Theorem]{Hin}.\ In this model structure the fibrations are the morphisms $f$ such that $f^1$ is surjective.\

Moreover, in \cite{CO} we proved that the category of $\Ain$-categories, linear over a field, is a fibrant category.\ 
The fibrations are the $\Ain$-functors $\F$ which are isofibrations and such that $\F^1$ is degree-wise surjective.\

We conclude by saying that we do not known if $\mbox{Alg}_{\infty}$ and $\aCat$ have coequalizers.\ 
This is an interesting question since, if it was true (at least in some cases i.e. along the cofibrations), one could try to prove that the category $\aCat$ (linear over a commutative ring) has a structure of cofibrant category.% (Theorem A provides candidates to be the cofibrant objects).

\section{Model Structures and Recognition Theorem}

We give two examples of model structures cofibrantly generated.\  
It will be crucial the following result which goes under the name of \emph{Recognition Theorem} and corresponds to \cite[Theorem 2.1.19]{Hov}:
\begin{thm}[Recognition Theorem]\label{coff}
Let $\C$ be a complete and cocomplete category with $\mathscr{W}$ a subcategory of $\C$, and $I$ and $J$ sets of maps of $\C$.\ There is a cofibrantly generated model structure on $\C$ with $I$ as the set of generating
cofibrations, $J$ as the set of generating trivial cofibrations, and $\mathscr{W}$ as the subcategory
of weak equivalences, if and only if the following conditions are satisfied:
\begin{itemize}
\item[1.] The subcategory $\mathscr{W}$ has the two out of three property and is closed under
retracts.
\item[2.] The domains of $I$ are small relative to $I$-cell.
\item[3.] The domains of $J$ are small relative to $J$-cell.
\item[4.] $J$-cell $\subset$ $\mathscr{W}\cap I\mbox{-cof}$.
\item[5.] $I$-inj $\subset$ $\mathscr{W}\cap J\mbox{-inj}$.
\item[6.] Either $\mathscr{W}\cap I\mbox{-cof}\subset J\mbox{-cof}$ or $\mathscr{W}\cap J\mbox{-inj}\subset I\mbox{-inj}$.
\end{itemize}
\end{thm}

We recall that $f\in I\mbox{-inj}$ if $f$ has the right lifting property with respect to any morphism in $I$.\
We say that $f\in I\mbox{-cof}$ if $f$ has the left lifting property for any $I$-injective morphism.

\begin{exmp}\label{unbCh}
We denote by $\mbox{Ch}(R)$ the category of unbounded chain complexes over a commutative ring $R$.\ 
The category $\mbox{Ch}(R)$ has a model structure whose weak-equivalences are the quasi-isomorphisms (see \cite[Definition 2.3.3 and Theorem 2.3.11]{Hov}).\
The set $I$ is given by the maps $i_n:\mathbb{S}^n\to\mathbb{D}^n$, for each $n\in\mathbb{Z}$, where $\mathbb{S}^{n}$ and $\mathbb{D}^{n}$ are the chain complexes so defined:
\begin{align*}
\mathbb{S}^{n}:=(R[n],0)
\end{align*}
and 
\begin{align*}
\mathbb{D}^{n}=\big(R[n+1]\oplus R[n],d_{\tiny\mbox{cone}(\Id_R)}:=\begin{pmatrix}
-d & 0 \\
\Id_R & d
\end{pmatrix}
\big).
\end{align*}
The set $J$ consists of $j_n:0\to \mathbb{D}^n$, for each $n\in\mathbb{Z}$.\ In this model structure the fibrations are the maps $f$ such that $f_n$ is surjective for all $n\in\mathbb{Z}$ and the cofibrant objects are the h-projective degreewise projective modules (see \cite[Theorem 9.6.1 (ii$’$) iff (v$’$)]{AFH}).
\end{exmp}

Before continuing, we denote by $K$ the Kontsevich category, which is the DG-category with two objects generated by the following morphisms:
\begin{align*}
\xymatrix{
1 \ar@(l,u)[]^{r_1} \ar[r]^f\ar@/^2pc/[r]^{r_{12}}&\ar@/^0.5pc/[l]^g2\ar@(r,u)[]_{r_2}
}
\end{align*}
with the relations:
\begin{itemize}
\item[1.] $d(r_{12})=r_2\cdot f + f\cdot r_1$.
\item[2.] $d(r_1)=g\cdot f-\Id_1$.
\item[3.] $d(r_2)=f\cdot g-\Id_2$.
\end{itemize}
Note that $K$ is a semi-free resolution of the DG-category $I$ (see \cite[3.7.6. Remark]{Dri}).\ 
$I$ is the DG-category with two objects and two closed morphisms of degree zero 
\begin{align}\label{int}
\xymatrix{
1\ar@/^/[r]^{j_{01}}&\ar@/^/[l]^{j_{10}}2
}
\end{align}
such that $j_{01}\cdot j_{10}=\Id_2$ and $j_{10}\cdot j_{01}=\Id_1$.\\
To prove that $K$ is semi-free, we consider the filtration:
\begin{align*} 
\xymatrix{
0\ar@{^(->}[r]&I_0 \ar@{^(->}[r]&I_1 \ar@{^(->}[r]&I_2 \ar@{^(->}[r]&I_3:=K.
}
\end{align*}
Here
\begin{itemize}
\item[1.] $I_0$ is the discrete category with two objects: 1 and 2.
\item[2.] $I_1$ is the DG-category freely generated by two closed generators $f$ and $g$
\item[3.] $I_2$ is the DG-category freely generated on $I_1$ by the generators $r_1$ and $r_2$, of degree 1, such that $d(r_1)=g\ast f -\Id_1$ and $d(r_2)=f\ast g -\Id_2$.
\item[4.] $I_3$ is the DG-category freely generated on $I_2$ adding the generator $r_{12}$ and taking the quotient by the DG-ideal $(d(r_{12})-r_2\ast f + f \ast r_1)$. 
\end{itemize}
We have the DG functors $\Psi_0:I_0\to I$, $\Psi_1:I_1\to I$ and $\Psi_2:I_2\twoheadrightarrow I$, $\Psi:K\twoheadrightarrow I$.\ 
In particular, $\Psi_2$ and $\Psi$ are surjective on the morphisms and $\Psi$ is a quasi-equivalence.\ In particular $\Psi(j_{01})=f$, $\Psi(j_{10})=g$ and $\Psi(r_1)=\Psi(r_2)=\Psi(r_{12})=0$.\

\begin{rem}
Note that we added $r_{12}$ in $I_3$, since 
\begin{align}\label{cicl}
r_2\ast f - f\ast r_1
\end{align} 
is a closed morphism which vanish in $I$ via $\Psi$.\ It means that, if we want to make $\Psi$ a quasi-equivalence then (\ref{cicl}) must be a coboundary in cohomology.\\
On the other hand, also 
\begin{align}\label{cicl2}
r_1\ast g - g \ast r_2 
\end{align} 
is a closed morphism whose image via $\Psi_2$ is 0.\ 
Nevertheless we do not need to add a new generator in $I_3$, since (\ref{cicl}) is the differential of the following morphism:
\begin{align*}
(g\ast r_{12}\ast g + r_1\ast g\ast r_2 - g\ast r_2\ast r_2 + r_1\ast r_1\ast g).
\end{align*}
\end{rem}

\begin{exmp}\label{Tabboz}
In \cite{Tab}, Tabuada proved that the category of DG-categories has a cofibrantly generated model structure whose weak-equivalences are the quasi-equivalences.\
The set of generating cofibrations $I$ consists of the DG-functors $Q$ and $S(n)$ described as follows.\ 
$Q$ is the (only) DG-functor 
\begin{align}
Q:\emptyset\to \A
\end{align}
from the empty set (which is the initial category of DG-cats and $\aCat$) to $\A$ which is the DG-category:
\[
\xymatrix{
3\ar@(r,u)[]_{\tiny\Id_3}
}
\]
with only one object, denoted by 3, and the identity.\
Fixed an integer $n$, $S(n)$ is the DG-functor
\begin{align*}
S(n):C(n)&\to P(n)\\
\mathbb{S}^{n-1}&\mapsto \mathbb{D}^{n-1}.
\end{align*}
Here $C(n)$ and $P(n)$ are the two DG-categories having two objects $8$, $9$ and $6$, $7$ and whose hom-spaces are given by  
\begin{align}
C(n)(8,9):=\mathbb{S}^{n-1}
\end{align}
and 
\begin{align}
P(n)(6,7):=\mathbb{D}^{n}.
\end{align}
The trivial cofibrations are generated by $J:=\mathcal{f}\mbox{$F$, $R(n)$}\mathcal{g}$.\ 
Fixed an integer $n$, the DG-functor $R(n):\B\to P(n)$ is defined as follows: 
\begin{align}
\xymatrix{
4\ar[d]&5 \ar[d]\\
6\ar[r]_{\mathbb{D}^{n}}&7
}
\end{align}
Here $\B$ is the category which has two objects 4 and 5 and no non trivial morphisms.\\
The functor $F:\A\to K$ is the DG-functor sending the object 3 of $\A$ to the object $1$ of $K$.\\
Note that, in this model structure the fibrations are the isofibrations (see Definition \ref{isofib}) which are degreewise surjective \cite[Proposition 1.13]{Tab2}.\
Moreover, every object is fibrant and the cofibrant objects are such that the hom-spaces are cofibrant in Ch($R$), see \cite[Proposition 2.3 (3)]{Toe}.
\end{exmp}

\section{Proof of Theorem A}

In this section we prove the main result using the Recognition Theorem \ref{coff}.\\
 
We define the category $K^{\tiny\Ain}$ to be the $\Ain$ semi-free resolution (according to \cite[\S5]{Orn2}) of the category $I$ (see (\ref{int})) defined as follows:
\begin{align*}
\xymatrix{
0\ar@{^(->}[r]& I_0 \ar@{^(->}[r]& I_1 \ar@{^(->}[r]& I_2 =:K^{\tiny\Ain}.
}
\end{align*}
\begin{itemize}
\item[0.] $I_0$ is the discrete (strictly unital) category with two objects: $1$ and $2$.
\item[1.] $I_1$ is the $\Ain$-category freely generated by the closed morphisms $j_{12}$ and $j_{21}$.
\item[2.] $I_2$ is the $\Ain$-category freely generated on $I_1$ adding the generators $r_1$ and $r_2$, such that 
\begin{itemize}
\item[i)] $m^1(r_1)= (\MDu; f,g) -1$. 
\item[ii)] $m^1(r_2)= (\MDu; g,f) -1$. 
\end{itemize}
\end{itemize}
Note that:
\begin{align*}
m^1_{K^{\tiny\Ain}}\big( (\MDu; r_1, f) - (\MDu; f, r_2 ) \big) &= \big( (\MDu; m^1(r_1), f) - (\MDu; f,m^1(r_2)) \big)\\
&= (\MCS ; f,g,f) - (\MCD; f,g,f )\\
&\not=0. 
\end{align*}
As we already noted, this is not the same in the case of DG-categories, since 
\begin{align*}
d_{K}( r_1\ast f - f \ast r_2 ) &= \big( d(r_1)\ast f) - (f \ast d(r_2)) \big)\\
&= (f\ast g\ast f) - (f\ast g \ast f )\\
&=0. 
\end{align*}
This is why we need to add the free generator $r_{12}$ in $K$ which is the coboundary of the cocycle $r_1\ast f - f \ast r_2$.\ 
Nevertheless $K^{\tiny\Ain}$ and $K$ are weakly equivalent since both are h-projective with splits unit and quasi-equivalent to $I$ (see \cite[Theorem 5.2]{Orn2} and \cite[Definition 3.4]{COS2}).\
It is important to say that, since $K^{\tiny\Ain}$ is h-projective, $I(n,m)=(R,0)$ is a free $R$-module (for any $n,m=\mathcal{f} 1,2 \mathcal{g}$) and they are quasi-equivalent, then $(K^{\tiny\Ain},m^1_{K^{\tiny\Ain}})$ is homotopy equivalent to the complex $(R,0)$.\
In particular, since $K^{\tiny\Ain}$ {has split unit}, it implies that the short exact sequence of graded complexes
\[
\xymatrix{
0\ar[r]&R\cdot 1_1\ar@{^(->}[r]&K^{\tiny\Ain}(1,1)\ar@{->>}[r]&{K}^{\tiny\Ain}(1,1)/R\cdot1_1\ar[r]&0.
}
\]
splits.\\

We define $F'$ to be the strict $\Ain$-functor from $\A$ to $K^{\tiny\Ain}$ such that $F'(3):=1$.

\begin{thm}\label{MS}
$\aCat_{\tiny\mbox{strict}}$ has a model structure cofibrantly generated by $J':=\mathcal{f} F',R(n)\mathcal{g}$ and $I$ (defined in Example \ref{Tabboz}).
\end{thm}

First we give two definitions and a Lemma.
\begin{defn}\label{isofib}
An $\Ain$-functor (resp. a DG-functor) $\F:\A\to\B$ is an \emph{isofibration} if, given $a\in\A$ and an isomorphism $g:\F(a)\to b$ in $H^0(\B)$, there exists an isomorphism $f:a\to a'$ in $H^0(\A)$, such that $\F^1(f)=g$.
\end{defn}

\begin{defn}\label{surj}
An $\Ain$-functor (resp. a DG-functor) $\F:\A\to\B$ is \emph{surjective} if, $\F^0$ is surjective (as a map of sets) and $\F^1$ is a surjective quasi-isomorphism of complexes.
\end{defn}

We denote by Surj the set of strict surjective $\Ain$-functors.\
The following Lemma which will be useful also to characterize the fibrations of the model structure provided by Theorem \ref{MS}.

\begin{lem}\label{RLP1}
A strict $\Ain$-functor $\mathsf{F}:\C\to\mathscr{D}$ has right lifting property with respect to
\begin{itemize}
\item[F1)] $R(n)$ if and only if $\mathsf{F}^1$ is surjective on the morphisms.
\item[F1$'$)] $S(n)$ if and only if $\mathsf{F}^1$ is surjective on the morphisms.
\item[F2)] $F$ if and only if $\mathsf{F}^1$ is an isofibration.
\item[F3)] $Q$ if and only if $\mathsf{F}^0$ is surjective (as a map of sets).
\end{itemize} 
\end{lem}

\begin{proof}
Suppose that $\mathsf{F}$ is surjective on the morphisms, consider the diagram below:
\[
\xymatrix{
\B\ar[d]_{R(n)}\ar[r]&\C\ar[d]^{\mathsf{F}}\\
P(n)\ar[r]_{T}\ar@{-->}[ur]^{\tilde{T}}& \mathscr{D}
}
\]
It is easy to find a $\tilde{T}$ since $P(n)$ is surjective on the morphisms and $P(n)$ is uniquely determined by the image of $T(R[1]\oplus R)$.\\
Suppose now that $\mathsf{F}$ has the right lifting property with respect to $R(n)$.\
For every morphism $g\in \mathscr{D}$ we take the functor 
\begin{align*}
T:R[1]\oplus R &\to \mathscr{D}\\
(a,b)&\mapsto a\cdot g + b\cdot dg .
\end{align*}
If there exists $\tilde{T}:\mathscr{P}(n)\to \mathscr{C}$ such that 
\begin{align*}
\mathsf{F}\big(\tilde{T}((a,b))\big)=T\big( (a,b)\big)=a\cdot g + b\cdot dg .
\end{align*}
Taking $\tilde{T}(1,0)\in\C$, we have $\mathsf{F}(\tilde{T}(1,0))=T((1,0))=g$, so $\mathscr{F}$ is surjective on morphisms.\\
\\
The proof of F1$'$) is the same of F1).\\
\\
To prove F2), note that every strict $\Ain$-functor $K^{\tiny\Ain}\to\mathscr{D}$ is (uniquely) determined by the 
image of the generators.\ Suppose that $\mathsf{F}$ has right lifting property with respect to $F$, so for every commutative diagram of the form:
\[
\xymatrix{
R\ar[d]_{F}\ar[r]&\C\ar[d]^{\mathsf{F}}\\
K^{\tiny\Ain}\ar[r]_{T}\ar@{-->}[ur]^{\tilde{T}}& \mathscr{D}
}
\]
there exists a lift $\tilde{T}$.\\
Suppose there exists $g:F(c)\to d \in H^0(\mathscr{D})$.\ It means that there exist $g^{-1}$, $h$, $t$ such that:
\begin{align*}
m_{\mathscr{D}}^2(g,g^{-1})=\Id+m_{\mathscr{D}}^1(h)
\end{align*}
and 
\begin{align*}
m_{\mathscr{D}}^2(g^{-1},g)=\Id+m_{\mathscr{D}}^1(t).
\end{align*}
Then we can take the functor $T:K^{\tiny\Ain}\to \mathscr{D}$ so defined:
\begin{align*}
T(j_{12})=g\mbox{ , } T(j_{21})=g^{-1}\mbox{ , }T(r_1)=t\mbox{ and } T(r_2)=h.
\end{align*}
Since $\mathsf{F}$ has a right lifting property, there exists $\tilde{T}:K^{\tiny\Ain}\to\C$.\ 
It implies that there exists an isomorphism $\tilde{T}(j_{12})=:f$ and an object $\tilde{T}(2)=:c'$ in $H^0(\C)$ 
such that $\mathsf{F}(f)=T(j_{12})=g$, and we are done.\ The proof of viceversa is similar.\\
\\
Now we prove F3).\ We consider the commutative diagram:
\[
\xymatrix{
\emptyset \ar[d]_{Q}\ar[r]&\C\ar[d]^{\mathsf{F}}\\
\A\ar[r]_{T}\ar@{-->}[ur]^{\tilde{T}}& \mathscr{D}
}
\]
since every functor from $\A$ is uniquely determined by the image of 3, it is easy to see that if there exists a lifting $\tilde{T}$ for every $T$ then $\mathsf{F}$ must be surjective on the objects.\ On the other hand, if $\mathsf{F}^0$ is surjective then we can find a $\tilde{T}$ for every $T$.
\end{proof}

\begin{proof}[Proof of Theorem \ref{MS}]
We want to use Theorem \ref{coff}.\ It is easy to verify item 1.\ 2.\ and 3.\\
\\
To prove item 5.\ and 6.\ we claim that Surj $=$ $I$-inj $=$ $J$-inj $\cap\mathscr{W}$.\ 
To prove of Surj $=$ $I$-inj we can use Lemma \ref{RLP1} item F1$'$) and F3).\ Note that it is the same of \cite[Lemme 1.11]{Tab2}, since the set of generating cofibrations $I$ is the same of Example \ref{Tabboz}.\ 
Now we want to prove $J$-inj $\cap\mathscr{W}$ $=$ Surj.\ 
If $f\in J\mbox{-inj}\cap\mathscr{W}$ then $f$ has right lifting property with respect to $R(n)$ so, by Lemma \ref{RLP1} item F1), it is surjective on the morphisms.\ For the item F3) of the same Lemma it is an isofibration.\ Since it is a quasi-equivalence it is surjective on the objects.\\
On the other hand if $f\in\mbox{Surj}$, then $R(n)$ has the right lifting property with respect to $f$ (see item F1) of Lemma \ref{RLP1}).\ Moreover $F$ has the right lifting property with respect to $f$ since $K^{\tiny\Ain}$ is semi-free (see \cite[Lemma 6.6]{Orn2}) and we are done.\\
\\
To verify item 4. we note that $I$-inj $=$ $\mathscr{W}\cap$ $J$-inj $\subset$ $J$-inj.\ It implies that $J$-cof $\subset$ $I$-cof and $J$-cell $\subset$ $J$-cof $\subset$ $I$-cof.\ 
So it remains to prove that, for every A$_{\infty}$-category $\mathscr{M}$ and $\star\in J'$, the strict A$_{\infty}$-functor $\mbox{inc}$, fitting the pushout diagram
\begin{align}
\xymatrix{
 \ar[d]_{\star}\ar[r]&\mathscr{M}\ar[d]^{\tiny\mbox{inc}}\\
\ar[r]&\mathscr{P},
}
\end{align}
is a quasi-equivalence.\\
First we consider the push-out diagram:
\begin{align}
\xymatrix{
\A \ar[d]_{F}\ar[r]^N&\mathscr{M}\ar[d]^{\tiny\mbox{inc}}\\
K^{\tiny\Ain}\ar[r]^{N'}&\mathscr{P}
}
\end{align}
We denote $N(3)\in\mathscr{M}$ by $z$.\ 
The push-out $\mathscr{P}$ is obtained by taking the disjoint union of $\mathscr{M}$ and $K^{\tiny\Ain}$ and gluing the object $z$ of $\mathscr{M}$ with the object $1$ of $K^{\tiny\Ain}$.\ See the following picture:

\begin{figure}[htbp]
\centering
\includegraphics[width=2\textwidth, height=.17\textheight, keepaspectratio]{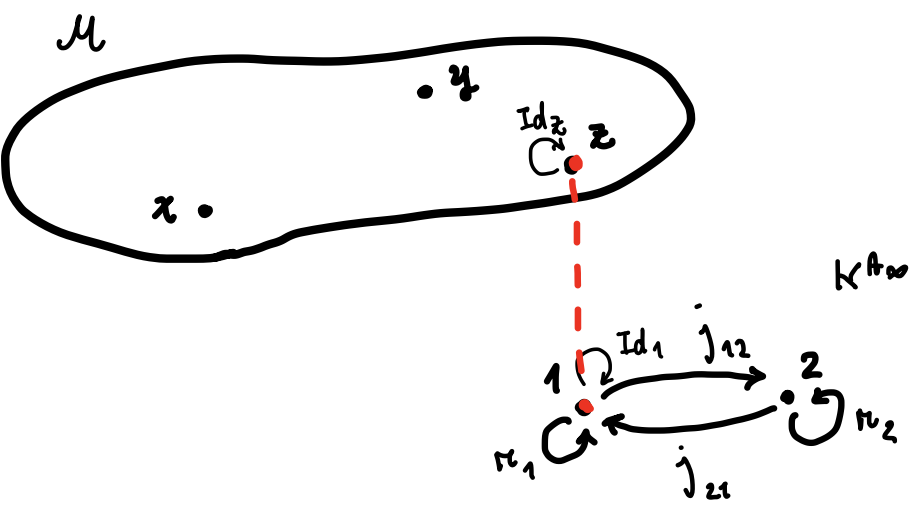}
\end{figure}
We have:
\begin{align}
\mathscr{P}(x,y):=\bigoplus_{m\ge0}\mathscr{P}^{(m)}(x,y),
\end{align}
where 
\begin{align*}
\mathscr{P}^{(m)}(x,y)= \underbrace{\mathscr{M}(z,y)\otimes \overline{K}^{\tiny\Ain}(z,z) \otimes \mathscr{M}(z,z)\otimes ...\otimes \overline{K}^{\tiny\Ain}(z,z)\otimes \mathscr{M}(x,z)}_\text{$m$ factors $\overline{K}^{\tiny\Ain}$}.
\end{align*}
Here $\overline{K}^{\tiny\Ain}(z,z)$ is the chain complex:
\begin{align}\label{quot}
\overline{K}^{\tiny\Ain}(z,z):={K}^{\tiny\Ain}(1,1)/R\cdot1_1.
\end{align}
Note that we take the quotient complex (\ref{quot}) because we "glue" the identity of the object $z\in \mathscr{M}$ with the identity of the object $1\in K^{\tiny\Ain}$.\
Note that the chain complex $(K^{\tiny\Ain}(1,1),m^1_{K^{\tiny\Ain}})$ is homotopy equivalent to $(R,0)$ and $\overline{K}^{\tiny\Ain}(z,z)$ is contractible.\ So
\begin{align*}
\mbox{inc}:\mathscr{M}(x,y) \to \mathscr{P}(x,y)
\end{align*}
is a quasi-isomorphism.\ It is also clear that $H^0(\mbox{inc})$ is essentially surjective since $N'(2)$ is quasi-isomorphic to $N'(1)=N(3)=z$.\\
On the other hand, we consider the push-out diagram:
\begin{align}
\xymatrix{
\B\ar[d]_{R(n)}\ar[r]^T&\mathscr{M}\ar[d]^{\tiny\mbox{inc}}\\
P(n)\ar[r]&\mathscr{P}
}
\end{align}
The category $\mathscr{P}$ is given by the disjoint union of $\mathscr{M}$ and $P(n)$, and by gluing the object $T(4)$ with the object $6$, and the object $T(5)$ with the object $7$.\ See the following picture:\\
\begin{figure}[htbp]
\centering
\includegraphics[width=2\textwidth, height=.17\textheight, keepaspectratio]{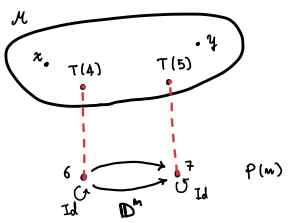}
\end{figure}
\\
We have:
\begin{align*}
\mathscr{P}(x,y)=\bigoplus_{m\ge0}\mathscr{P}^{(m)}(x,y),
\end{align*}
where 
\begin{align*}
\mathscr{P}^{(m)}(x,y)=\underbrace{\mathscr{M}(T(5),y)\otimes\mathbb{D}^n\otimes \mathscr{M}(T(5),T(4))\otimes ...\otimes \mathbb{D}^n\otimes \mathscr{M}(x,T(4))}_\text{$m$ factors $\mathbb{D}^n$},
\end{align*}
Since $\mathbb{D}^n$ is contractile then inc is a quasi-equivalence and we are done.
\end{proof}

\begin{cor}
The fibrations are the isofibrations $\mathsf{F}$ such that $\mathsf{F}^1$ are degreewise surjective.\
Every $\Ain$-category is fibrant.
\end{cor}
\begin{proof}
It follows directly from Lemma \ref{isofib}, every $\Ain$-category is fibrant since the terminal object of $\aCatSt$ is the category with one object and one (trivial) morphism $\A$ (defined Example \ref{Tabboz}).
\end{proof}

\begin{thm}
If $\A$ is a cofibrant $\Ain$-category then $\A(x,y)$ is a cofibrant object in Ch($R$).\
\end{thm}

\begin{proof}
The proof is the same as in the case of DG-categories (see \cite[Proposition 2.3 (3)]{Toe}) since they have the same set of generating cofibrations $I$ of Example \ref{Tabboz}.
\end{proof}

As in the case of DG-categories not all the h-projective $\Ain$-categories are cofibrant object in this model structure.\ For example the category $C(1)\otimes C(1)$ is not a cofibrant object.\ It is not hard to prove that the map $\emptyset\to C(1)\otimes C(1)$ does not have the right lifting property with respect to the trivial fibrations (see \cite[Exercise 14 4.]{Toe2}).\

%\newpage

\end{document}